\documentclass[11pt]{article}

\usepackage{amsmath, amsthm, amssymb, tikz, hyperref,MnSymbol}
\usepackage{amsfonts}
\usepackage{fullpage}
\usepackage[enableskew,vcentermath]{youngtab}
\usepackage{tikz-cd}

\newcommand\beq{\begin{equation}}
\newcommand\eeq{\end{equation}}
\newcommand\bce{\begin{center}}
\newcommand\ece{\end{center}}
\newcommand\bea{\begin{eqnarray}}
\newcommand\eea{\end{eqnarray}}
\newcommand\ba{\begin{array}}
\newcommand\ea{\end{array}}
\newcommand\ben{\begin{enumerate}}
\newcommand\een{\end{enumerate}}
\newcommand\bit{\begin{itemize}}
\newcommand\eit{\end{itemize}}
\newcommand\brr{\begin{array}}
\newcommand\err{\end{array}}
\newcommand\bt{\begin{tabular}}
\newcommand\et{\end{tabular}}

\newcommand\ul{\underline}

\renewcommand\S{{\mathcal S}}

\newcommand\ele{\varepsilon}

 \DeclareMathOperator\SYT{SYT}

\newcommand\x{{\mathbf x}}
\newcommand{\Q}{{\mathcal Q}}
\newcommand{\rot}{\mathrm{rot}}

\newcommand{\bbz}{\mathbb{Z}}

\newcommand{\bbc}{\mathbb{C}}

\DeclareMathOperator\Des{Des}
\DeclareMathOperator\cDes{cDes}
 
\DeclareMathOperator\id{id}

\newcommand{\ch}{\operatorname{ch}}

\newcommand{\jdt}{\operatorname{jdt}}
\newcommand{\ijdt}{\operatorname{ijdt}}

\newtheorem{theorem}{Theorem}[section]
\newtheorem{proposition}[theorem]{Proposition}
\newtheorem{lemma}[theorem]{Lemma}

\newtheorem{problem}[theorem]{Problem}

\theoremstyle{definition}
\newtheorem{definition}[theorem]{Definition}
\newtheorem{example}[theorem]{Example}

\theoremstyle{remark}

\newtheorem{remark}[theorem]{Remark}
\newtheorem{observation}[theorem]{Observation}

\title{On rotated Schur-positive sets}

\author{Sergi Elizalde~\thanks{Department of Mathematics, Dartmouth College, Hanover, NH 03755, USA. {\tt sergi.elizalde@dartmouth.edu}.
Partially supported by Simons Foundation grant \#280575 and NSA grant H98230-14-1-0125.
} \and Yuval
Roichman~\thanks{Department of Mathematics, Bar-Ilan University,
 Ramat-Gan 52900, Israel.  {\tt yuvalr@math.biu.ac.il}.}}

\date{}

\begin{document}
\maketitle

\begin{abstract}
The problem of finding Schur-positive sets of permutations,
originally posed by Gessel and Reutenauer, has seen some recent
developments. Schur-positive sets of pattern-avoiding permutations
have been found by Sagan et al and a general construction based on
geometric operations on grid classes 
has been given by the authors. In this paper we prove that
horizontal rotations of Schur-positive subsets of permutations are
always Schur-positive.
The proof applies a cyclic action on
standard Young tableaux
of certain skew shapes and a {\it jeu-de-taquin} type straightening algorithm.
As a consequence of the proof we obtain a notion of cyclic descent set on these tableaux, which is rotated by the cyclic action on them.
\end{abstract}

 \medskip


\section{Introduction}

For each $D \subseteq [n-1]=\{1,2,\dots,n-1\}$, define the fundamental {\em
quasisymmetric function}
\[
F_{n,D}(\x) := \sum\limits_{i_1\le i_2 \le \ldots \le i_n \atop {i_j <
i_{j+1} \text{ if } j \in D}} x_{i_1} x_{i_2} \cdots x_{i_n}.
\]
For a subset of permutations $B\subseteq \S_n$, let
\[
\Q(B):=\sum\limits_{\pi\in B} F_{n,\Des(\pi)},
\]
where $\Des(\pi)=\{i:\ \pi(i)>\pi(i+1)\}$ is the descent set of
$\pi$. If $B$ is a multiset of permutations in $\S_n$, we define
$\Q(B)$ analogously, by adding $F_{n,\Des(\pi)}$ as many times as
the multiplicity of $\pi$ in $B$.

We say that $B$ is {\em symmetric} if $\Q(B)$ is a symmetric
function. In this case, we say that $B$ is {\em Schur-positive} if
the expansion of $\Q(B)$ in the basis of Schur functions has
nonnegative coefficients. The problem of determining whether a
given subset of permutations is symmetric and Schur-positive was
first posed in~\cite{Gessel-Reutenauer}, see
also~\cite{Staley_thesis}, \cite{Gessel} and~~\cite{Stanley_problems}. The search for Schur-positive
subsets is an active area of research~\cite{Adin-R, Sh-Wachs,
Sagan_talk, CA2, CA, ER_pos_grids}.

For $J\subseteq [n-1]$, define the {\em descent class}
\[
D_{n,J}:=\{\pi\in \S_n:\ \Des(\pi)=J\}.
\]
and its inverse $D_{n,J}^{-1}=\{\pi:\pi^{-1}\in D_{n,J}\}$. It was shown in~\cite{Gessel} that inverse descent classes are Schur-positive.

Let $c$ be the $n$-cycle $(1,2,\dots,n)$, and let $C_n=\langle
c\rangle=\{c^k:\ 0\le k< n\}$, the cyclic subgroup of $\S_n$
generated by $c$. Given $\pi\in\S_n$, a permutation of the form $\pi c^k$ is called a {\em horizontal rotation} of $\pi$. Horizontal rotations played an important role in~\cite{ER_pos_grids}.

Any set $A\subseteq\S_{n-1}$ can be interpreted as a subset of $\S_n$ by identifying $\S_{n-1}$ with the set of the permutations in $\S_n$ that fix $n$. Then, $AC_n$ is the set of horizontal rotations of elements in $A$,
$$AC_n=\{\pi(k+1)\pi(k+2)\dots\pi(n-1)\,n\,\pi(1)\pi(2)\dots\pi(k)\,:\,\pi\in A,\, 0\le k< n\}.$$
Note in particular that all elements in $AC_n$ appear with
multiplicity one.

\smallskip

It was recently shown that horizontally rotated inverse descent
classes are always Schur-positive.

\begin{theorem}[{\cite[Theorem 7.1]{ER_pos_grids}}]
\label{thm:ER} For every $J\subseteq [n-2]$, the set
$D_{n-1,J}^{-1}C_n$ is Schur-positive.
\end{theorem}

The main result of this paper is a generalization of Theorem~\ref{thm:ER}, where $D_{n-1,J}^{-1}$ is replaced with an arbitrary Schur-positive set $A\subseteq \S_{n-1}$, stated in Theorem~\ref{thm:main}.

The proof involves a {\it jeu-de-taquin} type algorithm for nonstandard Young tableaux,
which is used to determine a $\bbz_n$-action on standard Young tableaux of certain skew shapes. This action
is analogous to the promotion cyclic
action on standard Young tableaux of rectangular shape introduced by
Rhoades~\cite{Rhoades}.
A byproduct of the proof is a notion of cyclic descent set on
standard Young tableaux of certain skew shapes, which is rotated by
the cyclic action on them.

\smallskip

The rest of the paper is organized as follows.
In Section~\ref{sec:main} we state the main result.
Section~\ref{sec:tools} introduces some tools needed for the proof, namely cyclic descent sets of permutations and rotated tableaux.
The proof of the main theorem appears in Section~\ref{sec:proof}. Section~\ref{sec:cDes} concludes with a discussion of cyclic descents of standard Young tableaux.

\section{Main Theorem}\label{sec:main}

We write $\lambda\vdash n$ to denote that $\lambda$ is a partition of $n$, and denote by $s_\lambda$ the Schur function indexed by $\lambda$. The following is our main result. When computing $\Q(A)$, we consider $A$ as a subset of $\S_{n-1}$.

\begin{theorem}\label{thm:main}
Let $n\ge2$. For every Schur-positive set $A\subseteq \S_{n-1}$,
the set $AC_n$ is Schur-positive. Additionally,
\begin{equation}\label{eq:main2}
\Q(AC_n)=\Q(A)s_1.
\end{equation}
\end{theorem}

In the statement of Theorem~\ref{thm:main}, the fact that $AC_n$
is Schur-positive  is an immediate consequence
Equation~\eqref{eq:main2}. Indeed, if $\Q(A)$ is Schur-positive,
then we obtain a non-negative expansion of $\Q(AC_n)$ in terms of
Schur functions using Pieri's rule~\cite[Theorem
7.15.7]{Stanley_ECII}, which states that for $\lambda\vdash n-1$,
$$s_\lambda s_1=\sum\limits_{\mu\vdash n \atop |\mu\setminus \lambda|=1} s_\mu.$$
Section~\ref{sec:proof} will be devoted to proving Equation~\eqref{eq:main2}.

An equivalent way to write Equation~\eqref{eq:main2} is
\begin{equation}\label{eq:main}
\Q(AC_n)\Q(\{\id\})=\Q(A)\Q(C_n),
\end{equation}
where $\id$ is the identity permutation in $\S_{n-1}$. To see this, note that $\Q(\{\id\})=s_{n-1}$ and $\Q(C_n)=s_n+s_{n-1,1}=s_1s_{n-1}$.

\begin{remark}
\begin{enumerate}
\item For arbitrary subsets $A\subseteq\S_{n-1}$,
Equation~\eqref{eq:main2} does not necessarily hold. For example,
if $n=4$ and $A=\{132\}\subset \S_3$, the left-hand side $\Q(AC_n)=2s_{2,2}$ is symmetric
and Schur-positive, but $\Q(A)s_1$ is not symmetric.

\item Vertical rotation, i.e., left multiplication of a
Schur-positive set $A\subseteq\S_{n-1}$ by $C_n$ does not
necessarily result in a Schur-positive set. For example, if
$A=\{3142, 1423\}\subset \S_4$, then $\Q(A)$ and $\Q(AC_5)$ are
Schur-positive, but $\Q(C_5 A)$ is not even symmetric.
\end{enumerate}
\end{remark}

\smallskip

We end this section with an equivalent formulation of the main theorems in terms of characters. Recall the Frobenius characteristic map, defined by \[
    \ch(\chi) \ = \ \frac{1}{n!} \, \sum_{\pi \in \S_n} \chi(\pi) p_\pi(x)
\]
where $\chi: \S_n \to \bbc$ is a class function, $p_\pi(x) =
p_\lambda (x)$ for every permutation $\pi$ of cycle type $\lambda
\vdash n$, and $p_\lambda(x)$ is a power sum symmetric function.

Using this terminology, Theorem~\ref{thm:main} is equivalent to
the following statement.

\begin{theorem}\label{thm:main_fine}
Let $\chi$ be an $\S_{n-1}$-character  and $A\subseteq \S_{n-1}$.
If $\Q(A)=\ch(\chi)$, then
\[
\Q(AC_n)=\ch(\chi\uparrow^{\S_n}).
\]
\end{theorem}

\section{Cyclic descents of rotated tableaux}
\label{sec:tools}

In this section we introduce some tools that will be used in the proof of Theorem~\ref{thm:main}.

\subsection{Standard Young tableaux and their rotations}
For $\lambda\vdash n$, denote by $\SYT(\lambda)$ the set of
standard Young tableaux (SYT for short) of shape $\lambda$, and define
$\SYT(\lambda/\mu)$ similarly for a skew shape $\lambda/\mu$. Let $s_\lambda$ and $s_{\lambda/\mu}$ denote the corresponding Schur functions. The descent set of a standard Young tableau $T$ is defined as
\begin{equation}\label{eq:DesT}
\Des(T):= \{ i: i+1 \text{ in a lower row than $i$ in $T$}\},
\end{equation}
where we use the English notation, in which row indices increase from top to bottom.

Recall the Robinson--Schensted correspondence, which associates to
each $\pi \in \S_n$ a pair $(P_\pi, Q_\pi)$ of standard Young
tableaux of the same shape $\lambda$, for some $\lambda \vdash n$.
The tableaux $P_\pi$ and $Q_\pi$ are called the {\em insertion}
and {\em recording} tableaux of $\pi$, respectively.
Inverting a permutation has the effect of switching the tableaux, namely,
$P_{\pi^{-1}}=Q_\pi$ and $Q_{\pi^{-1}}=P_\pi$ for all $\pi\in \S_n$.
The correspondence is a $\Des$-preserving bijection in the
following sense.

\begin{lemma}[{\cite[Lemma 7.23.1]{Stanley_ECII}}]\label{lem:RSKDes}
Let $\pi\in\S_n$. 
Then $\Des(\pi)=\Des(Q_\pi)$ and $\Des(\pi^{-1})=\Des(P_\pi)$.
\end{lemma}

The following is a well-known result of Gessel~\cite[Theorem 7.19.7]{Stanley_ECII}.

\begin{proposition}\label{prop:Gessel}
For every skew shape $\lambda/\mu$,
$$\sum_{T\in\SYT(\lambda/\mu)} F_{n,\Des(T)}=s_{\lambda/\mu}.$$
\end{proposition}

\medskip

\begin{definition}\label{ref:rotated}
A {\em rotated SYT} of size $n$ is a tableau on the letters $1,2,\dots,n$ where each letter appears exactly once, and entries are increasing along rows and columns with respect to the order
$$k+1<k+2<\dots<n<1<2<\dots<k,$$
for some $1\le k\le n$.
\end{definition}

In the rest of this paper, it will be convenient to consider the entries $1,2,\dots,n$ as elements of $\bbz_n$, so that $0$ is identified with $n$, and addition takes place modulo $n$.
For a SYT $T$ of size $n$ and an integer $k$, denote by
$k+T$ the tableau obtained by adding $k$ modulo $n$ to all its entries. Note that $R$ is a rotated tableau if and only if it is of the form $R=k+T$ for some $k$ and some SYT $T$.

\begin{example} The rotated tableau
$$R=\young(:::2,351,46)=2+\young(:::6,135,24)$$
is standard with respect to the order $3<4<5<6<1<2$.
\end{example}

\subsection{Cyclic descents of permutations and rotated SYT}

The {\em cyclic descent set} of a permutation was introduced by
Cellini~\cite{Cellini} and further studied in~\cite{Petersen,
Dilks}.

\begin{definition}\label{def_cyc_des_perm}
The {\em cyclic descent set} of $\pi\in \S_n$ is
$$\cDes(\pi)=\begin{cases} \Des(\pi) & \mbox{if }\pi(n)<\pi(1),\\
 \Des(\pi)\cup\{n\} & \mbox{if }\pi(n)>\pi(1).\end{cases}
$$
\end{definition}

Given a subset $D\subseteq [n]\simeq\bbz_n$, let $k+D=\{k+d : d\in D\}\subseteq[n]$, with addition modulo $n$.

\begin{observation}\label{obs1}
For every $\pi\in\S_n$ and $0\le k<n$,
\[
\cDes(\pi c^{-k})=k+\cDes(\pi).
\]
\end{observation}

A notion of cyclic descents for standard Young tableaux of
rectangular shapes was introduced by Rhoades~\cite{Rhoades}, see
also~\cite{PS}.
As in the case of permutations, the cyclic descent set respects a natural $\bbz_n$-action on the set of SYT of a given rectangular shape. This action, which was
horizontal rotation in the case of permutations (Observation~\ref{obs1}), is Sch\"utzenberger's promotion operation in the case of rectangular SYT. Also, in both cases, the cyclic descent set restricts to the regular descent set when the letter $n$ is ignored.

\medskip

The next definition, where again we identify $[n]\simeq\bbz_n$, extends this concept to rotated SYT.
For further discussion, see Section~\ref{sec:cDes}.

\begin{definition}\label{def_cyc_des_tab}
Let $R$ be a rotated SYT of size $n$. Define 
\begin{align*}
\cDes_\rot(R)&:=\{i\in [n]: i+1  \mbox{ is in a lower row than $i$ in $R$}\},\\
\Des(R)&:=\cDes_\rot(R)\cap[n-1].
\end{align*}
\end{definition}

\begin{remark}\label{rem:reading_word}
The reading word of a SYT $T$ is the permutation obtained by
reading the rows of $T$ from left to right and from bottom to top.
As an alternative to Definition~\ref{def_cyc_des_tab}, we could
have defined the cyclic descent set of a rotated SYT to be the
cyclic descent set of the inverse of the reading word of $T$. This is
equivalent to
\[
\cDes_\rot'(T):=\{i\in [n]: i+1  \mbox{ is strictly south of } i \mbox{
or in the same row and west of } i\},
\]

Figure~\ref{fig:cDes} shows a picture of the regions where $i+1$
has to be, relative to the location of $i$, for $i$ to be a cyclic
descent in these two definitions. Even though $\cDes_\rot$ and
$\cDes_\rot'$ do not coincide on rotated SYT in general, it can be
checked that they coincide on the kind of rotated SYT considered
in Section~\ref{sec:proof}. 
\end{remark}

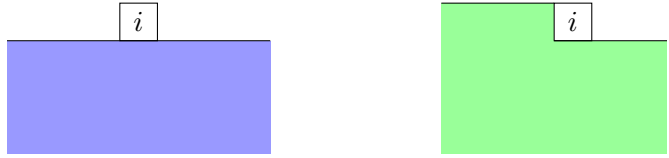
\begin{figure}[htb]
\centering
\begin{tikzpicture}[scale=.5]
\fill[blue!40!white] (-3,0) rectangle (4,-3);
\draw (0,0) rectangle (1,1);
\draw (0.5,0.5) node {$i$};
\draw (-3,0)--(4,0);
\end{tikzpicture}
\hspace{20mm}
\begin{tikzpicture}[scale=.5]
\fill[green!40!white] (0,1) rectangle (-3,-3);
\fill[green!40!white] (0,0) rectangle (3,-3); \draw (0,0)
rectangle (1,1); \draw (0.5,0.5) node {$i$}; \draw (-3,1)--(0,1);
\draw (0,0)--(3,0);
\end{tikzpicture}
\caption{Left: $i\in\cDes_\rot(T)$ if and only if $i+1$ is in the blue region. Right: $i\in\cDes_\rot'(T)$ if and only if $i+1$ is in the green region.}
\label{fig:cDes}
\end{figure}

\begin{observation}\label{obs2}
For every SYT $T$ of size $n$ and integer $k$
\[
\cDes_\rot(k+T)=k+\cDes_\rot(T).
\]
\end{observation}

\begin{example}\label{ex:1} For the tableaux
$$T=\young(:::6,135,24), \qquad 2+T=\young(:::2,351,46), \qquad 3+T=\young(:::3,462,51),
$$
we have $\cDes_\rot(T)=\{1,3,6\}$, $\cDes_\rot(2+T)= \{3,5,2\}= 2+\cDes_\rot(T)$, $\cDes_\rot(3+T)=\{4,6,3\}=3+\cDes_\rot(T)$.
\end{example}

\section{Proof of Theorem~\ref{thm:main}} \label{sec:proof}

In this section we will prove Equation~\eqref{eq:main2}. We start by showing that $\Q(AC_n)$ is completely determined by $\Q(A)$.

\begin{lemma}\label{lem:QAC}
Let $A$ and $A'$ be multisets of $\S_{n-1}$, and suppose that $\Q(A)=\Q(A')$. Then $\Q(AC_n)=\Q(A'C_n)$.
\end{lemma}

\begin{proof}
First, note that $\Q(A)=\Q(A')$ if and only if the distribution of $\Des$ is the same on $A$ as in $A'$. Indeed, the ``if" direction is trivial by definition, and the converse holds because the set
$\{F_{n,D}(\x):D\subseteq[n-1]\}$ of fundamental symmetric functions forms a basis of the vector space of homogeneous quasisymmetric functions of degree $n$ (see e.g.~\cite[Prop. 7.19.1]{Stanley_ECII}).

It now remains to show that the distribution of $\Des$ on $AC_n$ is completely determined by the distribution of $\Des$ on $A$. To see this, note that for $\sigma\in\S_{n-1}$, the descent set of a rotation
$\sigma c^{-k}$ depends only on the descent set of $\sigma$ and on $k$,
namely,
$$\Des(\sigma c^{-k})=(k+\Des(\sigma))\setminus\{n\}\cup\{k\}.$$

It follows that $\Q(AC_n)$ is completely determined by $\Q(A)$, and so $\Q(A)=\Q(A')$ implies $\Q(AC_n)=\Q(A'C_n)$.
\end{proof}

Now let $A\subseteq\S_{n-1}$ be a Schur-positive set. This means that we can write $\Q(A)=\sum_{\lambda\vdash n-1} c_\lambda s_\lambda$ for some coefficients $c_\lambda\ge0$.
Let $A'$ be the multiset consisting of the union of $c_\lambda$ copies of $A_\lambda$ for each $\lambda\vdash n-1$, where $A_\lambda\subseteq\S_{n-1}$ is any set satisfying $\Q(A_\lambda)=s_\lambda$. Then $\Q(A)=\Q(A')$, and so $\Q(AC_n)=\Q(A'C_n)$ by Lemma~\ref{lem:QAC}.
We will construct appropriate sets $A_\lambda$ and show that
$\Q(A_\lambda C_n)=s_\lambda s_1$ for all $\lambda$. It will then follow that
$$\Q(AC_n)=\Q(A'C_n)=\sum_\lambda c_\lambda Q(A_\lambda C_n)=\sum_\lambda c_\lambda Q(A_\lambda) s_1=\Q(A)s_1,$$ proving
Equation~\eqref{eq:main2}.

Fix $\lambda\vdash n-1$. Given $Q\in\SYT(\lambda)$, let
$\sigma^{-1}$ be its reading word, as described in
Remark~\ref{rem:reading_word}. It is easy to verify that $Q$ is
the insertion tableau of $\sigma^{-1}$ under RSK, thus the recording tableau of $\sigma$, hence
$\Des(\sigma)=\Des(Q)$ by Lemma~\ref{lem:RSKDes}.

Let us call $\sigma$ the {\em inverse reading word} of $Q$. Let
$A_\lambda$ be the set of all the permutations obtained as inverse
reading words of tableaux in $\SYT(\lambda)$. The map $Q\mapsto
\sigma$ that sends each tableau to its inverse reading word is
thus a $\Des$-preserving bijection from $\SYT(\lambda)$ to
$A_\lambda$. It follows that
$$\Q(A_\lambda)=\sum_{\sigma\in A_\lambda} F_{n,\Des(\sigma)}=\sum_{Q\in\SYT(\lambda)} F_{n,\Des(Q)}=s_\lambda,$$
using Proposition~\ref{prop:Gessel}.

Let $\lambda^\Box$ be the skew shape obtained from the Young
diagram of shape $\lambda$ by placing a disconnected box at its
upper right corner. For example, the tableaux in
Example~\ref{ex:1} have shape $(3,2)^\Box$. If
$P\in\SYT(\lambda^\Box)$, denote by $\delta(P)$ the entry it the
upper right box. Each $Q\in\SYT(\lambda)$ naturally corresponds to
a tableau $T\in\SYT(\lambda^\Box)$ with $\delta(T)=n$, obtained by
adding a box in the upper right corner with the entry $n$. If the
inverse reading word of $Q$ is $\sigma\in A_\lambda\subset
\S_{n-1}$, then the inverse reading word of $T$ is $\pi=\sigma n$,
the permutation obtained from $\sigma$ by adding $n$ as last
letter, which is an element of $A_\lambda$ when viewed as a subset
of $\S_n$. Additionally, for every $0\le k<n$, the inverse reading
word of $k+T$ is $\pi c^{-k}$, since its reading word
$c^k\pi^{-1}$ is obtained by adding $k$ to each entry of the
reading word $\pi^{-1}$ of $T$. Thus, the map $k+T\mapsto \pi
c^{-k}$ is a bijection from
$\{k+T:T\in\SYT(\lambda^\Box),\delta(T)=n\}$ to $\{\pi
c^{-k}:\pi\in A_\lambda\}$. Denote by
$$\varphi:\, \{\pi
c^{-k}:\pi\in A_\lambda\} \longrightarrow
\{k+T:T\in\SYT(\lambda^\Box),\delta(T)=n\}
$$
its inverse map,
and note that, for $\tau\in A_\lambda C_n$, $\varphi(\tau)$ is the
tableaux of shape $\lambda^\Box$ whose reading word is
$\tau^{-1}$.

\begin{example}
The reading word of the tableau $T$ in Example~\ref{ex:1} is $241356$, and so its inverse reading word is $\pi=314256$.
The reading word of $3+T$ is $514623$, and its inverse reading word is $\pi c^{-3}=256314$. In particular, $$\varphi(256314)=\young(:::3,462,51).$$
\end{example}

Clearly, for every $T\in \SYT(\lambda^\Box)$ with $\delta(T)=n$,
$\cDes_\rot(T)=\Des(T)\cup\{n\}=\Des(\pi)\cup\{n\}=\cDes(\pi)$.
Observations~\ref{obs2} and~\ref{obs1} imply now that, for every
$0\le k<n$,
$$\cDes_\rot(k+T)=k+\cDes_\rot(T)=k+\cDes(\pi)=\cDes(\pi c^{-k}).$$
In particular
\begin{equation}\label{eq4}
\Des(k+T)=\Des(\pi c^{-k}),
\end{equation}
and so $\varphi$ is a $\Des$-preserving bijection.

\medskip

In the rest of this section we will describe another $\Des$-preserving bijection
\begin{equation}\label{eq:bij}
\jdt:\, \{k+T:T\in\SYT(\lambda^\Box),\delta(T)=n\} \longrightarrow \{P\in\SYT(\lambda^\Box) : \delta(P)=k\}.
\end{equation}
Considering the composition
$$\{\pi c^{-k}:\pi\in A_\lambda\}\overset{\varphi}{\to}\{k+T:T\in\SYT(\lambda^\Box),\delta(T)=n\}\overset{\jdt}{\to}\{P\in\SYT(\lambda^\Box) : \delta(P)=k\}$$
and taking the union over $k$, we will obtain a $\Des$-preserving bijection between $A_\lambda C_n$ and $\SYT(\lambda^\Box)$, from where it will follow that
$$\Q(A_\lambda C_n)=\sum_{k=0}^{n-1} \sum_{\pi\in A_\lambda} F_{n,\Des(\pi c^{-k})}= \sum_{k=1}^{n} \sum_{P\in\SYT(\lambda^\Box)\atop \delta(P)=k} F_{n,\Des(P)}=\sum_{P\in\SYT(\lambda^\Box)} F_{n,\Des(P)}=s_{\lambda^\Box}=s_\lambda s_1,$$
using Proposition~\ref{prop:Gessel}, and the fact that the Schur
function indexed by the skew shape obtained by placing a partition
$\nu$ above and to the right of $\lambda$ is equal to
$s_{\lambda}s_{\nu}$.

\medskip

Fix  $0\le k<n$. Let $T\in\SYT(\lambda^\Box)$ with $\delta(T)=n$. Define $\jdt(k+T)$ to be the result of straightening the tableau $k+T$ by applying the following procedure, based on {\it jeu-de-taquin}. Initialize by setting $T_0=k+T$, and repeat the following step ---which we call {\em elementary step}--- until $T_0$ is a standard tableau (that is, its rows and columns are increasing):
\begin{itemize}
\item[($\ele$)]  Let $i$ be the minimal entry in $T_0$ for which the entry immediately above or to the left of it is larger than $i$ (such an entry exists because otherwise $T_0$ would be standard). Switch $i$ with the larger of these two entries, and let $T_0$ be the resulting tableau.
\end{itemize}
Let $\jdt(k+T)$ be the resulting standard tableau.

\begin{example}\label{ex:2}
For $k=3$, letting $T$ be the tableau $T$ in Example~\ref{ex:1}, we compute $\jdt(3+T)$ as follows:
\[
3+T=\young(:::3,462,51) \overset{\ele}{\mapsto} \young(:::3,412,56) \overset{\ele}{\mapsto}
\young(:::3,142,56) \overset{\ele}{\mapsto} \young(:::3,124,56)=\jdt(3+T).
\]
\end{example}

In Lemmas~\ref{lem:invertible} and~\ref{lem:jdt-preserves-Des} we will show that $\jdt$ is a $\Des$-preserving bijection as described in Equation~\eqref{eq:bij}. First, we introduce some terminology that will be used in the proofs.

An entry in a tableau that is smaller than an entry immediately above or to the left of it will be called {\em short}. Similarly, an entry that is larger than an entry below or to the right of it will be called {\em tall}.

In any tableau obtained along the process that takes $k+T$ to $P:=\jdt(k+T)$, the entries $1,2,\dots,k-1$ will be called the {\em moving} entries, while the entries $k+1,k+2,\dots,n$ will be called {\em non-moving}.
The moving entries in $k+T$ are the only ones that may be short, since, relative to the corresponding entries in the standard tableau $T$, they have decreased by $n-k$, while the remaining entries have increased by $k$.

If a tableau has the property that it contains no two moving
entries where the smaller one is weakly south and east of the
larger one, we say that this tableau restricted to the moving
entries is {\em standard}. Similarly for non-moving entries.

\begin{lemma}\label{lem:standard}
Let $Q$ be a tableau obtained in an intermediate step of the $\jdt$ process applied to $k+T$. Then $Q$ restricted to the moving entries $1,2,\dots,k-1$ is standard, and so is $Q$ restricted to the non-moving entries $k+1,k+2,\dots,n$.

Additionally, every elementary step performed by $\jdt$ when applied to $k+T$ switches one moving entry and one non-moving entry.
\end{lemma}

\begin{proof}
The first part of the lemma is clear when $Q=k+T$, since $T$ is standard and the relative order of the moving entries (resp. the non-moving entries) does not change when adding $k$ modulo~$n$.

Now let $Q$ be an intermediate tableau in the $\jdt$ process applied to $k+T$, and suppose that the restrictions of $Q$ to the moving entries and to the non-moving entries are standard. We will show that these properties are preserved when applying an elementary step to $Q$, and that this step switches a moving entry with a non-moving entry.

Let $i$ be the smallest short entry of $Q$, and note that by the
assumptions on $Q$ it has to be a moving entry. Then $\jdt$
switches $i$ with some larger entry $j$ above or to the left of
$i$. Since $Q$ restricted to moving entries is standard, the
location of $j$ implies that $j$ is a non-moving entry.

By definition of the $\jdt$ map, the entries $1,2,\dots,i-1$ in $Q$ form a left-justified standard Young tableau (otherwise $Q$ would have had a short entry smaller than $i$). Thus, after switching $i$ and $j$ in $Q'$, the resulting tableau $Q'$ still has the property of being standard when restricted to the moving entries.

To see that $Q'$ restricted to the non-moving entries is standard as well, note that the effect of an elementary step on the restriction of $Q$ to the non-moving entries is equivalent to a classical {\em jeu-de-taquin} slide in the south-east direction
(see~\cite[Appendix A1.2]{Stanley_ECII}), which preserves the standard property.
\end{proof}

\begin{lemma}\label{lem:invertible}
The map $\jdt$ is a bijection between $\{k+T:T\in\SYT(\lambda^\Box),\delta(T)=n\}$ and $\{P\in\SYT(\lambda^\Box) : \delta(P)=k\}$.
\end{lemma}

\begin{proof}
It is clear that both sets have the same cardinality $|\SYT(\lambda)|$, and that the image of any rotated tableau in the in the first set is in the second set.

To prove that the map is a bijection, we will describe the inverse map. Given $P\in\SYT(\lambda^\Box)$ with $\delta(P)=k$, consider the tableau $-k+P$, and define $\ijdt(-k+P)$ to be the tableau obtained by applying the following straightening procedure. Initialize by setting $P_0=-k+P$, and repeat the following elementary step until $P_0$ is a standard tableau:

\begin{itemize}
\item[($\ele'$)] Let $i$ be the maximal entry in $P_0$ for which the entry immediately below or to the right of it is smaller than $i$. Switch $i$ with the smaller of these two entries, and let $P_0$ be the resulting tableau.
\end{itemize}
Let $\ijdt(-k+P)$ be the resulting standard tableau.

\medskip

Let $T\in\SYT(\lambda^\Box)$ with $\delta(T)=n$, and let $P=\jdt(k+T)$.
We will show that
\[
\ijdt(-k+P)=T,
\]
that is, the inverse map of $\jdt$ is given by $\jdt^{-1}(P)=k+\ijdt(-k+P)$.
This fact is represented in the following diagram.
\begin{equation}\label{eq:diagram}
\begin{tikzcd}
k+T \arrow{r}{\jdt} & P \arrow{d}{-k} \\
T \arrow[swap]{u}{+k} & -k+P \arrow[swap]{l}{\ijdt}
\end{tikzcd}
\end{equation}

In any intermediate tableau obtained along the $\ijdt$ process applied to $-k+P$, and also in $T$, we define the {\em moving} entries to be $n+1-k,n+2-k,\dots,n-1$. Note that the moving entries in $T$ (resp. $P$) become the moving entries in $k+T$ (resp. $-k+P$) when adding (resp. subtracting) $k$ modulo $n$.

When applying $\jdt$ to $k+T$, the process starts by moving entry $1$ (if it is short) in the north and/or west direction (by switching it with non-moving entries) until it is no longer short. Then the algorithm does the same to entry $2$, and so on, until finally it moves entry $k-1$ until it is not longer short.

In $-k+P$, the moving entries are the only ones that may be tall. When applying $\ijdt$ to $-k+P$, the process starts by moving entry $n-1$ in the south and/or east direction until it is no longer tall, then it moves entry $n-2$ similarly, and so on, until it finally moves entry $n-k+1$.

We will show that each elementary step in the $\jdt$ process that takes $k+T$ to $P$ is reversed by an elementary step of $\ijdt$. This is illustrated the following diagram, where we use $\ele_j$ (resp. $\ele'_j$) to denote an elementary step of $\jdt$ (resp. $\ijdt$) that moves the entry $j$.
\begin{equation}
\begin{tikzcd}
k+T \arrow{r}{\ele_1} & .  \arrow{d}{\pm k} &[-2.5em] \cdots &[-2.5em] .  \arrow{r}{\ele_1} \arrow{d}{\pm k} & . \arrow{r}{\ele_2} \arrow{d}{\pm k} & \ \cdots\ \arrow{r}{\ele_{k-2}} & . \arrow{r}{\ele_{k-1}} \arrow{d}{\pm k}  & .  \arrow{d}{\pm k} &[-2.5em] \cdots &[-2.5em] . \arrow{r}{\ele_{k-1}} \arrow{d}{\pm k}  & P=\jdt(k+T) \arrow{d}{-k} \\
T=\ijdt(-k+P) \arrow[swap]{u}{+k} & . \arrow{u} \arrow[swap]{l}{\ele'_{n-k+1}} &[-2.5em] \cdots &[-2.5em] . \arrow{u}& . \arrow{u} \arrow[swap]{l}{\ele'_{n-k+1}}  &  \ \cdots\  \arrow[swap]{l}{\ele'_{n-k+2}} & . \arrow{u} \arrow[swap]{l}{\ele_{n-2}} & . \arrow{u} \arrow[swap]{l}{\ele'_{n-1}} &[-2.5em] \cdots &[-2.5em] . \arrow{u}& -k+P \arrow[swap]{l}{\ele'_{n-1}}
\end{tikzcd}
\end{equation}

Consider one elementary step $\ele$ in this $\jdt$ process. Let $Q$ be
the tableau at that moment, and let $k-i$ be the moving entry
about to be switched at that step (that is, $k-i$ is the minimal
short entry of $Q$). Let $a$ and $\ell$ be the entries above and
to the left of $k-i$ in $Q$, respectively, if they exist
(otherwise, define them to be $0$), and suppose that $\ell>a$ (the
case $\ell<a$ is symmetric). Since $k-i$ is short, we have
$\ell>k-i$, and the current elementary step $\ele$ switches $\ell$ with
$k-i$.

Let $Q'=\ele(Q)$ be the resulting tableau. 
It suffices to show that one elementary step $\ele'$ of $\ijdt$ applied to
$-k+Q'$ reverses this switch, that is $\ele'(-k+Q')=-k+Q$. The entries $k-i$ and $\ell$ in $Q$
become $n-i$ and $\ell-k$ in $-k+Q'$. Since $\ell$ was not a
moving entry in $Q$ (by Lemma~\ref{lem:standard}), $\ell-k$ is not
a moving entry in $-k+Q'$, but $n-i$ is, so we have $n-i>\ell-k$,
which means that $n-i$ is a tall entry in $-k+Q'$. We claim that
$n-i$ is the maximal tall entry in $-k+Q'$. This is because any
entry $n-j>n-i$ corresponds to a moving entry $k-j>k-i$ in $Q'$
and in $k+T$. Because moving entries are treated by $\jdt$ in
increasing order, neither $k-j$ nor the entries weakly south and
east of it in $k+T$ have been moved by $\jdt$ so far. Thus, the
entry $n-j$ in $-k+Q'$ and the entries weakly south and east of it
are the same as in $T$, which is standard. In particular, $n-j$ is
not tall in $-k+Q'$.

We have shown that $n-i$ is the maximal tall entry in $-k+Q'$.

If there is no entry below $\ell$ in $Q$, then there is no entry below $n-i$ in $-k+Q'$, so $\ele'$ switches $n-i$ and $\ell-k$ as desired.
Otherwise, let $s$ be the entry below $\ell$ in $Q$. If $s$ is a moving entry we are done, because moving entries are never switched with each other, by Lemma~\ref{lem:standard} adapted to $\ijdt$. If $s$ is not a moving entry, then the fact that $Q$ restricted to non-moving entries is standard implies that $s>\ell$. It follows that the corresponding entries in $-k+Q'$ satisfy $s-k>\ell-k$, which implies that $\ele'$ switches $n-i$ with $\ell-k$ in this case as well.
\end{proof}

\begin{example}\label{ex:jdt} For $T$ as in Examples~\ref{ex:1} and~\ref{ex:2} and $P=\jdt(3+T)$, applying $\ijdt$ to $-3+P$ we have
$$\begin{array}{rcccccl}
3+T=\young(:::3,462,51) & \overset{\ele_1}{\mapsto} & \young(:::3,412,56) & \overset{\ele_1}{\mapsto} &
\young(:::3,142,56) & \overset{\ele_2}{\mapsto} & \young(:::3,124,56)=\jdt(3+T)=P\medskip \\ 
\upmapsto +3&&&&&&\quad\downmapsto -3\medskip\\
T=\ijdt(-3+P)=\young(:::6,135,24)&\overset{\ele'_4}{\leftmapsto}&\young(:::6,145,23)&\overset{\ele'_4}{\leftmapsto}&\young(:::6,415,23)&\overset{\ele'_5}{\leftmapsto}&\young(:::6,451,23)=-3+P
\end{array}$$
\end{example}

To complete the proof of Theorem~\ref{thm:main}, it remains to show that $\jdt$ preserves the descent set. 
Recall from Definition~\ref{def_cyc_des_tab} that, for $1\le i\le n-1$, we have $i\in\Des(k+T)$ if and only $i+1$ is in a lower row than $i$ in $k+T$.

\begin{lemma}\label{lem:jdt-preserves-Des}
For every $0\le k<n$ and every $T\in\SYT(\lambda^\Box)$ with $\delta(T)=n$, we have
\[
\Des(\jdt(k+T))=\Des(k+T).
\]
\end{lemma}

\begin{example} For the tableau $T$ in Example~\ref{ex:1}, we have
$$\begin{array}{ccc@{\quad}l}
2+T=\young(:::2,351,46)& \mapsto& \jdt(2+T)=\young(:::2,135,46), & 2+\cDes_\rot(T)=\{3,5,2\}=\Des(\jdt(2+T));\\
3+T=\young(:::3,462,51)& \mapsto&\jdt(3+T)=\young(:::3,124,56), & 3+\cDes_\rot(T)=\{4,6,3\}=\Des(\jdt(3+T))\cup\{6\};\\
4+T=\young(:::4,513,62)& \mapsto&\jdt(4+T)=\young(:::4,135,26), & 4+\cDes_\rot(T)=\{5,1,4\}=\Des(\jdt(4+T)).
\end{array}
$$
\end{example}

\begin{proof}[Proof of Lemma~\ref{lem:jdt-preserves-Des}]
We will show that for every $i\in[n-1]$, $i\in\Des(k+T)$ if and only if $i\in\Des(\jdt(k+T))$. The case $k=0$ is trivial, so we assume that $k\neq 0$. Since $\delta(k+T)=k$, it is clear that $k-1\notin\Des(k+T)$ but $k\in\Des(k+T)$, and the same holds for the tableau $\jdt(k+T)$. Thus, it suffices to consider the cases $i<k-1$ and $i>k$.
Recall that the entries smaller than $k$ are called moving entries in $k+T$.

\medskip
\noindent{\bf Case $i<k-1$}. The entries $i$ and $i+1$ in $k+T$ correspond to entries $n-k+i$ and $n-k+i+1$ in $T$, respectively, and so we either have that $i+1$ is strictly south and weakly west of $i$ (in which case $i\in\Des(k+T)$), or $i+1$ is weakly north and strictly east of $i$ (in which case $i\notin\Des(k+T)$). These possibilities are shown in Figure~\ref{fig:i}. Similarly, since $\jdt(k+T)$ is standard, the relative position of $i$ and $i+1$ is also given by one of the above two possibilities.

\begin{figure}[htb]
\centering
\begin{tikzpicture}[scale=.5]
\fill[green!40!white] (0,0) rectangle (-4,-3);
\fill[blue!40!white] (0,0) rectangle (4,3);
\fill[red!40!white] (0,0) rectangle (4,-3);
\draw (-1,0) rectangle (0,1);
\draw (-0.5,0.5) node {$i$};
\draw (-4,0)--(4,0);
\draw (0,-3)--(0,3);
\draw (2,2) node {$i+1$ here};
\draw (2,1) node {if $i\notin\Des$};
\draw (-2,-1) node {$i+1$ here};
\draw (-2,-2) node {if $i\in\Des$};
\draw (2,-1) node {$i+1$ can't};
\draw (2,-2) node {be here};
\end{tikzpicture}
\caption{The possible locations of $i+1$ relative to $i$ in $k+T$ and $\jdt(k+T)$}
\label{fig:i}
\end{figure}
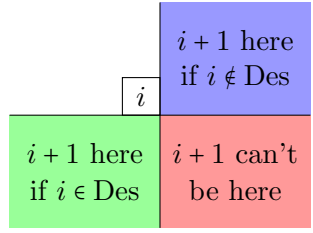

Suppose for contradiction that $i\notin \Des(k+T)$ but $i\in\Des(\jdt(k+T))$ (the case
$i\in \Des(k+T)$ but $i\notin\Des(\jdt(k+T))$ is symmetric with respect to the diagonal of $T$).
In other words, $i+1$ is weakly north and strictly east of $i$ in $k+T$, but strictly south and weakly west of $i$ in $\jdt(k+t)$. Consider the paths (as sequences of cells, each of them north or west from the previous one) that the entries $i$ and $i+1$ follow when $\jdt$ is applied to $k+T$, and recall that first $\jdt$ moves $i$ along its path, and afterwards it moves $i+1$.  The relative location of $i$ and $i+1$ before and after applying $\jdt$ forces these two paths to intersect, and so there must be a cell $C$ such that the path of $i$ leaves $C$ by going north, while the path of $i+1$ enters $C$ by going west. Let $a$ be the entry in $C$ right after $i$ leaves this cell, and let $b$ be the entry immediately northeast of $C$ at that time. Let us first argue that $a,b$ are both non-moving entries. This is clear for $a$ because it is switched with $i$, and the only switches involve a moving entry and a non-moving one  by Lemma~\ref{lem:standard}. If $b$ was moving, then we would have $i<b$, since the current tableau restricted to moving entries is standard by Lemma~\ref{lem:standard}, but then the fact that $i+1$ must be strictly south and weakly east of $b$ (implied by the fact that $i+1$ enters $C$ later in the $\jdt$ process) would contradict Lemma~\ref{lem:standard}.

Before $a$ and $i$ were switched, $a$ was immediately west of $b$, and so $a<b$, since the tableau restricted to the non-moving entries is standard by Lemma~\ref{lem:standard}.
But this contradicts that $i+1$ arrives at cell $C$ by moving west, since at that point the entry north of $i+1$ is $b$ and the entry west of $i+1$ is $a$.

\medskip

\noindent{\bf Case $i>k$}.  When we apply $\jdt$ to $k+T$, the effect on the subtableau consisting of the non-moving entries is the same as that of applying classical {\em jeu-de-taquin} slides in the south-east direction. For example, when $\jdt$ moves $1$ until this entry is no longer short, the algorithm restricted to entries larger than $k$ corresponds to a {\em jeu-de-taquin} slide into the position that $1$ occupied originally.
It is known that classical {\em jeu-de-taquin} preserves the descent set (see \cite[Lemma 3.2]{Doran}), and so the descents between entries larger than $k$ are preserved by $\jdt$.
\end{proof}

\section{Cyclic descents of SYT}\label{sec:cDes}

Rhoades introduced a notion of cyclic descents
on SYT of rectangular shapes, having the property that the
$\bbz_n$-action on SYT of fixed rectangular shape by promotion rotates their cyclic descent
sets~\cite{Rhoades}. In our notation, the promotion operation can be described as
$$T\mapsto\jdt(1+T).$$
It was noticed by Rhoades that the promotion operator does
not determine a $\bbz_n$-action on the set of SYT of a general shape.
In this section we extend the concept of cyclic descents to SYT of shape $\lambda^\Box$, for any partition $\lambda\vdash n-1$, and describe a $\bbz_n$-action on these tableaux that rotates their cyclic descent sets.
This extension is a consequence of the proof presented in
Section~\ref{sec:proof}.

\medskip

Recall from diagram~\eqref{eq:diagram} that, for every $P\in\SYT(\lambda^\Box)$,
$$\jdt^{-1}(P)=\delta(P)+\ijdt(-\delta(P)+P),$$
where $\delta(P)$ is the entry in the upper right box of $P$.

\begin{definition}\label{def:cDes}
For $P\in\SYT(\lambda^\Box)$, define its {\em cyclic descent set} by
$$\cDes(P):=\cDes_\rot(\jdt^{-1}(P)).$$
\end{definition}

\begin{example}
As in Example~\ref{ex:jdt}, for
$$P=\young(:::3,124,56),$$
we have $$\jdt^{-1}(P)=3+\ijdt(-3+P)=\young(:::3,462,51),$$
and so $\cDes(P)=\{3,4,6\}$.
\end{example}

\begin{proposition}\label{prop:action}
\begin{enumerate}
\item For every $P\in\SYT(\lambda^\Box)$,
$$\cDes(P)\cap[n-1]=\Des(P).$$
\item The map
\[
P\mapsto \jdt\left(1+\jdt^{-1}(P)\right) \]
determines a $\bbz_n$-action on $\SYT(\lambda^\Box)$, which rotates the cyclic descent sets of the
tableaux.
\end{enumerate}

\end{proposition}

\begin{proof}
1. Using Lemma~\ref{lem:jdt-preserves-Des} and Definitions~\ref{def_cyc_des_tab} and~\ref{def:cDes},
$$\Des(P)=\Des(\jdt^{-1}(P))=\cDes_\rot(\jdt^{-1}(P))\cap[n-1]=\cDes(P)\cap[n-1].$$

\noindent 2.  By Observation~\ref{obs2} and Definition~\ref{def:cDes},
$$
\cDes(\jdt(1+\jdt^{-1}(P))=\cDes_\rot(1+\jdt^{-1}(P))=1+\cDes_\rot(\jdt^{-1}(P))=1+\cDes(P).\qedhere
$$
\end{proof}

The above proposition raises the natural problem of finding a unified approach to other shapes. The concept of a cyclic descent satisfying the properties in Proposition~\ref{prop:action} can be generalized as follows.

\begin{definition}
Let $B$ be a set of combinatorial objects carrying a descent set map $\Des:B\to2^{[n-1]}$.
A {\em cyclic descent extension} for $B$
is a pair $(\psi,\cDes)$ where $\psi$ is a $\bbz_n$-action on $B$
and $\cDes$ is a map from $B$ 
to $2^{[n]}$ such that, for all $T\in B$, 
\begin{itemize}
\item[(i)] $\cDes(T)\cap[n-1]=\Des(T)$,
\item[(ii)] $\cDes(\psi(k)T)=k+\cDes(T)$ for all $k\in\bbz_n$.
\end{itemize}
\end{definition}

\medskip

Two examples of cyclic descent extensions that were known before this paper are the following:
\begin{itemize}
\item Take $B=\S_n$, let $\cDes$ be as in Definition~\ref{def_cyc_des_perm}, and 
let $\psi$ be right multiplication by $c^{-1}$ (that is, horizontal rotation). See Observation~\ref{obs1}.
\item Take $B=\SYT(r^{n/r})$, the set of SYT of given rectangular shape,
let $\cDes$ be as defined by Rhoades~\cite{Rhoades}, and let $\psi$ be the promotion operation on SYT.
\end{itemize}

In this paper we have introduced two new examples of cyclic descent extensions:
\begin{itemize}
\item Let $B$ be the set of rotated SYT of a given shape, let $\cDes=\cDes_\rot$ be as in Definition~\ref{def_cyc_des_tab}, and let $\psi$ be addition modulo $n$. See Observation~\ref{obs2}.
\item Take $B=\SYT(\lambda^\Box)$, let $\cDes$ be as in Definition~\ref{def:cDes}, and let
 $\psi$ be the map in Part 2 of Proposition~\ref{prop:action}.
\end{itemize}

\begin{problem}
For which skew shapes $\lambda/\mu$ does there exist a cyclic descent extension for $B=\SYT(\lambda/\mu)$?
\end{problem}

This problem, which was posed in an early version of the paper, is currently being addressed in~\cite{AER,ARR}.

\end{document}